\documentclass{amsart}

\usepackage{amssymb}

\newtheorem{theorem}{Theorem}[section]
\newtheorem{proposition}[theorem]{Proposition}
\newtheorem{lemma}[theorem]{Lemma}
\newtheorem{corollary}[theorem]{Corollary}

\theoremstyle{definition}
\newtheorem{definition}[theorem]{Definition}

\theoremstyle{remark}
\newtheorem{remark}[theorem]{Remark}
\newtheorem{claim}{Claim}

\numberwithin{equation}{section}

\begin{document}

\title{Differentiability, porosity and doubling in metric measure spaces}

\author{David Bate}
\address{Mathematics Institute\\
Zeeman Building\\
University of Warwick\\
Coventry \\
CV4 7AL UK}
\curraddr{}
\email{D.S.Bate@Warwick.ac.uk}

\author{Gareth Speight}
\address{Mathematics Institute\\
Zeeman Building\\
University of Warwick\\
Coventry \\
CV4 7AL UK}
\curraddr{}
\email{G.Speight@Warwick.ac.uk}

\thanks{This work was done under the supervision of David Preiss and supported by EPSRC}

\date{\today}

\begin{abstract}
We show if a metric measure space admits a differentiable structure then porous sets have measure zero and hence the measure is pointwise doubling. We then give a construction to show if we only require an approximate differentiable structure the measure need no longer be pointwise doubling.
\end{abstract}

\maketitle

\section{Basic Definitions}
Let $(X,d)$ be a complete separable metric space equipped with a locally finite Borel measure $\mu$. We begin by recalling the notion of a differentiable structure on $(X,d,\mu)$ introduced by Cheeger \cite{Cheeger}. We call a metric measure space which admits such a structure a differentiability space and define a weaker notion of approximate differentiability space.

\begin{definition}
If $U \subset X$ is Borel measurable and $\varphi\colon X \to \mathbb{R}^{n}$ is Lipschitz we say $(U,\varphi)$ is a chart on $(X,d,\mu)$ of dimension $n \in \mathbb{N}$. A map $f\colon X \to \mathbb{R}$ is differentiable with respect to a chart $(U, \varphi)$ of dimension $n$ at $x_{0} \in U$ if there exists a unique derivative $df(x_{0}) \in \mathbb{R}^{n}$ (depending on the chart) such that
\[\limsup_{X\ni x \to x_{0}} \frac{ |f(x)-f(x_{0})-df(x_{0})\cdot (\varphi(x)-\varphi(x_{0}))|}{d(x,x_{0})}=0.\]

We say $(X,d,\mu)$ is a differentiability space if there exists a countable decomposition of $X$ into charts so that any Lipschitz $f\colon X\to \mathbb R$ is differentiable at almost every point of every chart.

We say $(X,d,\mu)$ is an approximate differentiability space if the above holds but with the limit replaced by an approximate limit. We say a function $g\colon X \to \mathbb{R}$ has approximate limit $l \in \mathbb{R}$ at $x_{0} \in X$ if for every $\varepsilon > 0$,
\[ \lim_{r\downarrow 0} \frac{\mu \{x \in B(x_{0},r): |g(x)-l|>\varepsilon \}}{\mu(B(x_{0},r))}=0.\]
We observe that since $(X,d)$ is separable for almost every $x_{0}$ we have $\mu(B(x_{0},r))>0$ for all $r>0$.
\end{definition}

Examples of differentiability spaces include Euclidean spaces, the Heisenberg group \cite{Cheeger} and Laakso spaces \cite{Laakso}. As Keith remarks, rectifiable sets in metric spaces are approximate differentiability spaces with Hausdorff measure \cite{Keith}. Keith has shown that if the measure $\mu$ on $X$ is globally doubling, namely balls have finite positive measure and there exists $C\geq 1$ such that $\mu(B(x,2r))\leq C\mu(B(x,r))$ for all $x \in X$ and $r>0$, then the notions of differentiability space and approximate differentiability space coincide \cite{Keith2}.

\begin{remark}
Note the fact that chart maps in differentiability spaces can be taken to be Lipschitz follows automatically from the definition of differentiability space, even if we don't insist they are Lipschitz in the definition of a chart.
\end{remark}

\begin{definition}
For $\eta>0$ we say that a set $S\subset X$ is $\eta$-\emph{porous} at $x_0\in S$ if there exist $x_n\to x_0$ with
\[d(x_n,S) > \eta d(x_n,x_0)\]
and that such an $x_n$ is a witness of porosity for $x_0$.  Further we call $S$ \emph{porous} if it is porous at each of its points and define for each $x_0\in S$
\[\eta(x_0)=\sup\{\eta : S \text{ is } 2\eta \text{ porous at } x_0\}\]
so that for every $x_0\in S$, $S$ is $\eta(x_0)$-porous at $x_0$. Finally, for $r>0$, we define the function $\sigma_r\colon S\to\mathbb R$ by
\[\sigma_r(x_0)=\sup\{d(x,S)<r : x \text{ is a } \eta(x_0) \text{-witness of porosity for } x_0\}.\]
\end{definition}

\begin{remark}
Observe that for a porous set $S$ and $C>0$, $\eta(x_0)> C$ if and only if there exists an $m\in\mathbb N$ and for every $n\in\mathbb N$ an $x\in X$ with $d(x,x_0)<1/n$ and 
\[d(x,S)>(2C+1/m)d(x,x_0).\]
This final condition is open in $x_0$ and so an arbitrary union over $x\in X$ followed by a countable intersection over $n$ and union over $m$ shows that $x_0\mapsto \eta(x_0)$ is a Borel function.  For any $r>0$ a similar decomposition shows that $\sigma_r$ is also Borel.
\end{remark}

\begin{definition}
We say $\mu$ is pointwise doubling at $x \in X$ if
\[\limsup_{r \downarrow 0} \frac{\mu (B(x,2r))}{\mu (B(x,r))} < \infty.\]
Note an equivalent condition is obtained if 2 is replaced by any other enlargement factor greater than one.
\end{definition}

We will see porous sets in differentiability spaces necessarily have measure zero and thus the underlying measure is pointwise doubling almost everywhere. We then show any approximate differentiability space which gives measure zero to porous sets is necessarily a differentiability space. Finally we show, by construction of an example, that the measure underlying an approximate differentiability space need not be pointwise doubling in a set of positive measure.

\section{Porous sets in differentiability spaces}
We begin with a characterisation of the uniqueness of derivatives with respect to a chart.

\begin{lemma}\label{lem:discuniq}
Let $(U,\varphi)$ be an $n$-dimensional chart in a metric measure space $(X,d,\mu)$ and $x_0\in U$.  Given $f\colon X\to \mathbb R$ a derivative of $f$ with respect to $(U,\varphi)$ at $x_0$ (if one exists) is unique if and only if there exists a $\lambda>0$ and $X\ni x_m\to x_0$ so that for any $v\in\mathbb R^n$
\begin{equation}\label{eq:uniqueness}\lim_{m\to\infty}\max_{0\leq i<n}\frac{|(\varphi(x_{mn+i})-\varphi(x_0))\cdot v|}{d(x_{mn+i},x_0)}\geq \lambda \|v\|.\end{equation}
In particular, uniqueness of derivatives at $x_0$ depends only on the chart and is independent of the function we differentiate.
\end{lemma}

\begin{remark}
Notice that if derivatives at $x_0$ are unique, then for any $f\colon X\to \mathbb R$ with derivative $df(x_0)$ at $x_0$ we have
\[\liminf_{m\to\infty}\max_{0\leq i<n}\frac{|f(x_{mn+i})-f(x_0)|}{d(x_{mn+i},x_0)}\geq \lambda\|df(x_0)\|.\]
\end{remark}

\begin{remark}
Note that any chart $(U,\varphi)$ with respect to which Lipschitz functions have a non-unique derivative almost everywhere may be decomposed into a finite union of new charts $(U_{i},\varphi_{i})$ with respect to which derivatives exist and are unique. The chart map $\varphi_{i}$ in each of these new charts is a subcollection (possibly empty) of the coordinates of the original map $\varphi$.
\end{remark}

\begin{proof}
First suppose that the conclusion holds and that a function $f\colon X\to \mathbb R$ has two derivatives $D$ and $D' \in \mathbb R^n$.  Then by the definition of derivative and the triangle inequality we see that
\[\limsup_{M\ni x\to x_0}\frac{|(\varphi(x)-\varphi(x_0))\cdot (D-D')|}{d(x,x_0)}=0\]
and so $\|D-D'\|=0$.

To show the converse observe that, since $\varphi$ is Lipschitz, the left hand side of \eqref{eq:uniqueness} is homogeneous and continuous in $v$, regardless of our choice of $x_m$.  Therefore it suffices to show the existence of a sequence so that \eqref{eq:uniqueness} gives positive values to non-zero $v$. For this we find a sequence $x_m$ so that for every $0\leq i < n$
\[\lim_{m\to\infty}\frac{\varphi(x_{mn+i})-\varphi(x_0)}{d(x_{mn+i},x_0)}\]
exists and form a basis of $\mathbb R^n$.

Suppose no such sequence exists.  Then for every sequence so that the above limits exist the limit vectors must belong to some $n-1$ dimensional subspace $V$ of $\mathbb R^n$.  Then for $v^\perp\neq 0$ orthogonal to $V$ we have
\[\limsup_{X\ni x\to x_0}\frac{|(\varphi(x)-\varphi(x_0))\cdot v^\perp|}{d(x,x_0)}=0\]
and so for any $f\colon X\to \mathbb R$ with derivative $df(x_0)$ at $x_0$, $df(x_0)+v^\perp$ is another derivative, contradicting our hypothesis.
\end{proof}

We may now prove our first result.

\begin{theorem}
Any porous set $S$ in a Lipschitz differentiability space is null.
\end{theorem}

\begin{proof}
It suffices to show that for an $n$-dimensional chart $(U,\varphi)$ any porous $S\subset U$ is null. Given such an $S$, $\sigma_r$ is a Borel function and so there exists a $S'\subset S$ with $\mu(S')\geq\mu(S)/2$ and a sequence $r_m\to 0$ monotonically so that for every $x_0\in S'$ there exist $\eta(x_0)$-witnesses $y_m$ with
\begin{equation}\label{eq:regions}r_{m+1}<d(y_m,S)/2<3d(y_m,S)/2 <r_m.\end{equation}
In particular for any $y\in B(y_m,d(y_m,S)/2)$ we have
\[r_{m+1}<d(y,S)<r_m.\]
We first partition these regions of $X\setminus S$ into $n+1$ cases that will allow us to define non-differentiable Lipschitz functions.  For $0\leq i <n+1$ let $R_i$ be the set of $x\in X$ with
\[d(x,S)\in \bigcup_{m\in\mathbb N}(r_{m(n+1)+i+1},r_{m(n+1)+i}).\]
Then, for $y\not\in S$, define the Lipschitz map $g_y\colon X\to \mathbb R$ by
\[g_y(x)=\max(d(y,S)/2-d(y,x),0);\]
a ``cone'' centred at $y$ of radius $d(y,S)/2$. Finally for $0\leq i< n+1$ define $f_i$ to be the 1-Lipschitz map
\[f_i=\sup\{g_y : y \text{ an } \eta(x_0) \text{-witness of some } x_0\in S',\ y \in R_i\}.\]
Note that by the above observation this satisfies $f_i(x)=0$ for any $x\not\in R_i$. We may now partition $S'$ into sets of non-differentiability for some $f_i$.

For almost every $x_0\in S'$ let $x_m\to x_0$ be as in the conclusion of lemma \ref{lem:discuniq} for $x_0$. Then by the pigeon-hole principle, there exists an $i$ and $m_k\to\infty$ so that for every $k\in\mathbb N$ and $0\leq j < n$ either $d(x_{m_kn+j},S)=0$ or $x_{m_kn+j}\not\in R_i$. Thus $f(x_{m_kn+j})=f(x_0)=0$ and so by lemma \ref{lem:discuniq}, $df_i(x_0)=0$ if it exists.

In addition, since $x_0\in S'$, we may find for any $m$ an $\eta(x_0)$-witnesses $x$ of $x_0$ satisfying \eqref{eq:regions} and so $\limsup |f(x)-f(x_0)|/d(x,x_0) \geq \eta(x_0)/2$.
This would require $df(x_0)\neq 0$ and so $f_i$ cannot be differentiable at $x_0$.

We have therefore decomposed $S'$ into a finite union of sets of non-differentiability and a set of non-uniqueness. Therefore $S'$ and hence $S$ are null.
\end{proof}

\begin{remark}
If $S \subset \mathbb{R}^{n}$ is porous then the distance function $d(x,S)$ provides a very easy example of a Lipschitz map differentiable nowhere on $S$. However in the more general setting such a function may be a coordinate chart and so the more involved construction given above is required.  To see this take $S\subset \mathbb R$ be a non-empty, closed null set.  We show that $\mathbb R$ with coordinate chart $d(x,S)$ is a Lipschitz differentiability space.

Let $(a,b)$ be a connected component of $\mathbb R\setminus S$, so that at most one of $a$ or $b$ is infinite. Then every $x_0\in((a+b)/2,b)$ has a neighbourhood where $d(x,S)-d(x_0,S)=x-x_0$ and so if a Lipschitz function $f$ has Euclidean derivative $df$ at $x_0$, $f$ also has derivative $df$ with respect to $d$ and vice versa.  Similarly every $x_0\in (a,(a+b)/2)$ has a neighbourhood where $d(x,S)-d(x_0,S)=-(x - x_0)$ and so such an $f$ has derivative $-df$ with respect to $d$. Thus any Lipschitz function has a unique derivative almost everywhere with respect to $d(x,S)$.
\end{remark}

\begin{corollary}
Theorem 3.6(iv) \cite{Preiss} gives a decomposition $X=P\cup S\cup N$ where $S$ is $\sigma$-porous, $N$ is null and $\mu$ is pointwise doubling at any point of $P$.  Thus a Lipschitz differentiability space is pointwise doubling almost everywhere. Thus, by the standard proof, $(X,d,\mu)$ is a Vitali space.
\end{corollary}

\begin{corollary}
Any subset $S$ of a Lipschitz differentiability space $X$ is also a Lipschitz differentiability space with respect to the charts obtained by restricting charts for $X$ to $S$.  Moreover, the derivative of any Lipschitz $f\colon S\to \mathbb R$ agrees with the derivative of any extension $f\colon X\to\mathbb R$ almost everywhere in $S$.
\end{corollary}

\begin{proof}
Let $(U,\varphi)$ be a chart in $X$ and suppose that $S\subset U$.  Any Lipschitz map $f\colon S\to \mathbb{R}$ extends to a Lipschitz map on $X$ and so we have a candidate for the derivative of $f$ at almost every point of $S$; we must just check for uniqueness.  As seen above this is a property of $S$ and non-uniqueness occurs at any $x_0$ such that
\[\limsup_{S\ni x\to x_0}\frac{|(\varphi(x)-\varphi(x_{0}))\cdot D|}{d(x,x_0)}=0\]
for some $D\in \mathbb{R}^n$ with $D\neq 0$.  However for almost every such $x_0$ the uniqueness of derivatives in $X$ provides a sequence $x_m\to x_0$ in $X$ and a $\eta>0$ so that 
\[ \frac{|(\varphi(x_{m})-\varphi(x_{0}))\cdot D|}{d(x_{m},x_{0})}>\eta.\]
Thus, for sufficiently large $m$, the ball of radius $\eta\ d(x_{m},x_0)/2\operatorname{Lip}\varphi$ around $x_{m}$ does not belong to $S$ so that the set of such $x_0$ is porous and so null.  Thus $S$ is a Lipschitz differentiability space.
\end{proof}

Since any differentiability space is also an approximate differentiability space we see admitting an approximate differentiable structure and porous sets being null is a necessary condition for a space to admit a differentiable structure. The following proposition shows this condition is also sufficient.

\begin{proposition}
Suppose $(X,d,\mu)$ is a metric measure space in which porous sets have measure zero. Then for almost every $x_0 \in X$ we have that whenever $g \colon X\to \mathbb{R}$ is a Lipschitz function such that
\[\operatorname*{aplim}_{x\to x_0} g(x)/d(x,x_0)=0\]
then
\[\lim_{x\to x_0} g(x)/d(x,x_0)=0.\]
\end{proposition}

\begin{proof}
In Proposition 3.5 \cite{Keith2} Keith proves a variant of this result under the different hypothesis that $\mu$ is globally doubling. An analysis of his proof shows the weaker condition that for almost every $x$ and all $\varepsilon > 0$
\[ \liminf_{y \to x} \frac{\mu(B(y,\varepsilon d(y,x)))}{\mu(B(x,2d(y,x)))}>0\]
suffices. As in \cite{Preiss} let $\gamma(\mu,x,r,\delta)$ be the supremum of the set
\[\{s>0: \exists z \text{ such that } d(x,z)+s\leq r \text{ and } \mu(B(z,s))\leq \delta \mu(B(x,r))\}\]
and define the upper porosity of $\mu$ at $x$ by
\[\overline{\mathrm{por}}(\mu,x)=\lim_{\delta \downarrow 0} \limsup_{r \downarrow 0} \frac{\gamma(\mu,x,r,\delta)}{r}.\]
By Theorem 3.6 \cite{Preiss} if porous sets have measure zero then $\overline{\mathrm{por}}(\mu,x)=0$ for almost every $x$. This immediately implies the condition which suffices for Keith's proof.
\end{proof}

\section{A non doubling approximate differentiability space}

It is easy to see the measure in approximate differentiability spaces need not give measure zero to porous sets. Indeed, $\mathbb{R}^{2}$ with one dimensional Hausdorff measure restricted to a line is an example of an approximate differentiability space which gives full measure to a porous set. We show, by construction of a counterexample, that the measure underlying an approximate differentiability space need not even be pointwise doubling. 

The space we construct is based on Laakso's spaces \cite{Laakso}. We start with a simple base space, in which distinct points are not connected by (continuous) paths, equipped with a non doubling measure. We then take the product of this space with an interval and Lebesgue measure. We identify different points in the base space at different heights on the interval so as to make the quotient space path connected and we equip it with the path metric. Since we cannot move horizontally inside the base space and have to travel vertically to move between points the quotient acquires a differentiable structure somewhat like Euclidean space. Due to the choice of identifications the product measure is still non doubling at almost every point.

For clarity we will work with a pseudometric, which satisfies the requirements of a metric except distances between distinct points may be $0$, rather than a metric on the quotient space. The following remark explains how these are equivalent.

\begin{remark}
Suppose a set $M$ is equipped with a pseudometric $d'$. Then we can define balls, Borel sets and so on in $(M,d')$ as in any metric space. Suppose $(M,d')$ is complete, separable and equipped with a locally finite Borel measure $\mu$. Then the notions of differentiability space and approximate differentiability space also make sense for $(M,d')$.

Define an equivalence relation on $(M,d')$ by $x \sim y$ if $d'(x,y)=0$. The set of equivalence classes $\widetilde{M}:=\{[x] : x \in M\}$ is a metric space with well defined distance $d([x],[y])=d'(x,y)$. Let $i \colon M \to \widetilde{M}$ denote the projection $x \mapsto [x]$. Notice $i(B_{d'}(x,r))=B_{d}([x],r)$ and $i^{-1}(B_{d}([x],r))=B_{d'}(x,r)$ for $x \in X$ and $r>0$. It is then easy to see that $(\widetilde{M},d)$ is complete and separable and $i_{\star}(\mu)$ is a locally finite Borel measure on $(\widetilde{M},d)$. If $\mu$ on $(M,d')$ is non pointwise doubling almost everywhere then $i_{\star}(\mu)$ on $(\widetilde{M},d)$ is also non pointwise doubling almost everywhere. Further, if $(M,d',\mu)$ is a (approximate) differentiability space then $(\widetilde{M},d,i_{\star}(\mu))$ is also a (approximate) differentiability space. This follows because Lipschitz functions on $(M,d')$ are constant on equivalence classes which implies, using uniqueness of derivatives in $(M,d',\mu)$, that the derivative of $f$ is constant on equivalence classes, hence gives rise to a function defined almost everywhere on $\widetilde{M}$.
\end{remark}

\begin{definition}
We define the set
\[M := \{ (a_{1}, a_{2}, a_{3}, \ldots ) : a_{i} \in \{1, \ldots, i \} \} \]
and equip $M$ with the metric $d_{M}(a,b)=1/2^k$ where $k=\min (i : a_{i}\neq b_{i})$.

If $k\geq 1$ we call a set of the form
\[ \{ (c_{1}, c_{2}, c_{3}, \ldots, c_{k}, a_{k+1}, a_{k+2}, \ldots ) : a_{i}\in \{1, \ldots, i\} \: \forall i \geq k+1 \}\]
for $c_{1}, c_{2}, \ldots c_{k}$ fixed an island of level $k$. Note there are $k!$ islands of level $k$. 

Notice $M$ is an ultrametric space. That is, the distance $d_{M}$ satisfies the strong triangle inequality $d_{M}(x,y)\leq \max (d_{M}(x,z),d_{M}(z,y))$ for all $x, y, z \in M$. This implies any two balls in $M$ are either disjoint or one is contained inside the other. Balls in $M$ are islands and the Borel sigma algebra on $M$ consists of countable unions of disjoint islands and points. Hence we can define a Borel probability measure $\nu$ on $M$ which gives the same measure $\nu_{k}=1/k!$ to any island of level $k$.
\end{definition}

Notice that if $a \in M$ then $B_{M}(a,1/2^{k})$ is the island of level $k$ containing $a$. Hence, for any $a \in M$,
\[ \nu(B_{M}(a,1/2^{k+1}))/\nu(B_{M}(a,1/2^k))=\nu_{k+1}/\nu_{k} \to 0\]
as $k \to \infty$ so the measure $\nu$ is not pointwise doubling. 

However, since any two balls are either disjoint or one is contained inside the other, $(M,d_{M},\nu)$ is a Vitali space and consequently the density theorem holds in $(M,d_{M},\nu)$. 

The space $M$ is compact. To see this note that given a sequence of elements of $M$ we can find a subsequence for which each coordinate is eventually constant. This immediately implies the subsequence converges.

Let $I=[0,1]$. We now introduce a pseudometric on $M \times I$ which gives rise to the quotient space described. 

\begin{definition}
For integer $k \geq 0$ define $I_{k}:=\{ n/2^{k} : n \in \mathbb{N}, \: 0\leq n \leq 2^k \} \subset I$. Define a pseudometric $d_{p}$ on $M \times I$ by 
\[d_{p}((x,t),(y,s)):= \inf \{ |t-u|+|u-s| : u \in I_{k},\: d_{M}(x,y) < 1/2^{k} \}.\]
We call subsets of the form $M \times \{ n/2^k \}$, for some $n \in \mathbb{N}$, jump levels of size $k$ and denote the open ball with centre $(x,t)$ and radius $r>0$ in $(M \times I, d_{p})$ by $B_{p}((x,t),r)$. The height of a point $(x,t)$ is its coordinate $t\in I$.
\end{definition}

Notice that geometrically the pseudometric $d_{p}$ corresponds to the shortest path joining points and moving only vertically, except at jump levels of size $k$ where we can jump horizontally inside islands of level $k$.

Balls $B_{p}((x,t),r)$, $(x,t) \in M \times I$ and $r>0$, are countable unions of overlapping rectangles. The rectangles in this union are centred on jump levels and are of the form
\[ B_{M}(x,1/2^{k}) \times (u-r+|t-u|,u+r-|t-u|)\]
where $u \in I_{k}$ and $|t-u|<r$. Notice this implies $\nu \times \mathcal{L}^{1}$ is a Borel measure on $(M \times I,d_{p})$.

We will also use the maximum metric $d_{\infty}((x,t),(y,s))=\max (d_{M}(x,y),|t-s|)$ on $M \times I$. We show $d_{p}((x,t),(y,s))\leq 3 d_{\infty}((x,t),(y,s))$ for all $(x,t), (y,s) \in M \times I$. Fix $(x,t)$ and $(y,s)$ in $M \times I$. If $x=y$ then the inequality is obvious as both distances are equal to $|t-s|$. Suppose not and $d_{M}(x,y)=1/2^{k+1}$ for some fixed $k$. Then $x$ and $y$ belong to the same island of level $k$ but distinct islands of level $k+1$. We can find $u \in I_{k}$ with $|t-u|\leq 1/2^{k+1}$. Hence,
\[ d_{p}((x,t),(y,s)) \leq 1/2^{k+1}+1/2^{k+1} + |s-t| \leq 3d_{\infty}(x,y).\]
Notice this implies $i\colon (M \times I,d_{\infty}) \to (M \times I, d_{p})$ is Lipschitz. Since $M$ and $I$ are separable, the collection of $d_{\infty}$ Borel sets is the smallest sigma algebra containing products of Borel sets in $M$ and $I$. 

Since $(M,d_{M})$ and $I$ are compact it follows the product space $(M\times I, d_{\infty})$ is also compact. Since $i\colon (M\times I, d_{\infty}) \to (M\times I, d_{p})$ is continuous and surjective we see $(M\times I, d_{p})$ is compact, hence complete and separable.

\section{Non pointwise doubling}

\begin{theorem}
The measure $\nu \times \mathcal{L}^{1}$ on $(M \times I, d_{p})$ is non pointwise doubling almost everywhere. Indeed, for almost every $(x,t) \in M \times I$,
\[ \limsup_{r \downarrow 0} \frac{\nu \times \mathcal{L}^{1} (B_{p}((x,t),4r)) }{\nu\times \mathcal{L}^{1} (B_{p}((x,t),r))} = \infty. \]
\end{theorem}

\begin{proof}
The main idea is that, on small scales, rectangles centred on larger jump levels have much greater measure than those with similar height and centred on smaller jump levels. Let
\[ E_{k}= \{t \in I : |t-u|\geq 1/2^{k+2} \  \forall u \in I_{k}\}.\]
We claim $\nu \times \mathcal{L}^{1}(M \times \limsup E_{k})=1$. To see this it suffices to show that for any fixed $k$ the set $\cap_{l \geq k} (E_{l})^{c}$ has measure $0$. We notice that if $t \in I$, $|t-n/2^l|< 1/2^{l+2}$ for a fixed $n \in \mathbb{N}$ and there exists $m \in \mathbb{N}$ such that $|t-m/2^{l+1}|<1/2^{l+3}$ then $|t-n/2^l|<1/2^{l+3}$. Hence, repeatedly applying this observation, we see $\cap_{l \geq k} (E_{l})^{c} = I_{k}$ so has measure $0$. 

Fix $(x,t) \in M \times \limsup E_{k}$ and an increasing sequence $n_{k} \to \infty$ such that $t \in E_{n_{k}}$ for every $k$.  We cannot reach a jump level of size $n_{k}$ within vertical distance strictly less than $1/2^{n_{k}+2}$ of $(x,t)$. Hence 
\[ B_{p}((x,t),1/2^{n_{k}+2}) \subset B_{M}(x,1/2^{n_{k}+1}) \times (t-1/2^{n_{k}+2},t+1/2^{n_{k}+2})\]
which implies
\[ \nu \times \mathcal{L}^{1}(B_{p}((x,t),1/2^{n_{k}+2})) \leq 1/2^{n_{k}+1}\nu_{n_{k}+1}.\]

We note that for each fixed $k$ there exists $u \in I_{n_{k}}$ with $|t-u|\leq 1/2^{n_{k}+1}$. Hence 
\[ B_{M}(x,1/2^{n_{k}}) \times (u-1/2^{n_{k}+1},u+1/2^{n_{k}+1}) \subset B_{p}((x,t),1/2^{n_{k}})\]
which implies 
\[\nu \times \mathcal{L}^{1} (B_{p}((x,t),1/2^{n_{k}})) \geq 1/2^{n_{k}} \nu_{n_{k}}.\]

Putting these inequalities together we observe

\[ \frac{\nu \times \mathcal{L}^{1}(B_{p}((x,t),1/2^{n_{k}}))}{\nu \times \mathcal{L}^{1} (B_{p}((x,t),1/2^{n_{k}+2}))} \geq  2 \frac{\nu_{n_{k}}}{\nu_{n_{k+1}}} \to \infty\]
as $k \to \infty$. Hence $\nu\times \mathcal{L}^{1}$ is not pointwise doubling at $(x,t)$.
\end{proof}

\section{Approximate Differentiability}

We now show $(M \times I, d_{p}, \nu \times \mathcal{L}^{1})$ is an approximate differentiability space.

\begin{theorem}
The space $(M \times I, d_{p}, \nu \times \mathcal{L}^{1})$ is an approximate differentiability space. The structure consists of the single chart $(M \times I, h)$ and for each Lipschitz map and almost every $(x_0,t_0) \in M\times I$ the derivative of $f$ at $(x_0,t_0)$ is given by
\[ df (x_{0},t_{0}) := \lim_{u \to 0} \frac{f(x_{0},t_{0}+u)-f(x_{0},t_{0})}{u}\]
\end{theorem}

\begin{remark}
While it will take effort to prove existence of derivatives with respect to $h$ it is not hard to see uniqueness. If $a_{1}, a_{2} \in \mathbb{R}$ are possible candidates for derivatives of a given function at $(x_{0},t_{0}) \in M \times I$ then by the triangle inequality,
\[ \operatorname{aplim} \frac{|(a_{1}-a_{2})(t-t_{0})|}{d_{p}((x,t),(x_{0},t_{0}))}=0.\]
To see this implies $a_{1}=a_{2}$ it suffices to observe that for each $r>0$ the points $(x,t)\in B_{p}((x_{0},t_{0}),r)$ satisfying $|t-t_{0}|>r/2$ comprise at least half the measure of $B_{p}((x_{0},t_{0}),r)$ and satisfy $|t-t_{0}|>r/2>d_{p}((x,t),(x_{0},t_{0}))/2$.
\end{remark}

To prove the theorem we show balls in $(M \times I, d_{p})$ are on small scales well approximated by at most three, not necessarily disjoint, rectangles.

\begin{lemma}
Fix $\varepsilon > 0, x \in M$ and $t \in I\setminus \cup_{k=1}^{\infty} I_{k}$. Choose $r>0$ sufficiently small so that $(t-r,t+r)\cap I_{k} \neq \varnothing$ implies $k>0$ and $\nu_{k+1}/\nu_{k} < \varepsilon /2$. Then there exist at most three rectangles $R_{i}\subset B_{p}((x,t),r)$ of the form
\[ R_{i}=B_{M}(x,1/2^{k_{i}}) \times (t_{i}-r_{i},t_{i}+r_{i}),\]
where $t_{i} \in I_{k_{i}}$, $r_{i}=r-|t-t_{i}|>0$, $k_{1}= \min \{ k \in \mathbb{N} : (t-r/2,t+r/2) \cap I_{k} \neq \varnothing \}$ and $k_{2}, k_{3}< k_{1}$, if defined, such that
\[ \frac{\nu \times \mathcal{L}^{1}( B_{p}((x,t),r) \setminus \cup_{i} R_{i})}{\nu \times \mathcal{L}^{1} (B_{p}((x,t),r))}  < \varepsilon.\]
\end{lemma}

\begin{proof}
The main idea is that rectangles centred on larger jump levels and relatively close to the centre of a ball contain most of the measure. Fix $(x,t)\in M \times I, \varepsilon > 0$ and $r>0$ as in the statement of the lemma. Define $k_{1}$ corresponding to the largest jump level we can reach within height $r/2$ above or below $(x,t)$. Since $t\not\in \cup_{k=1}^{\infty} I_k$ this intersection is a single point, $t_1$ say.

Since $(t-r/2,t+r/2)\cap I_{k_{1}-1}=\varnothing$ and elements of $I_{k_{1}-1}$ are spaced only a distance $1/2^{k_{1}-1}$ apart we deduce $r<1/2^{k_{1}-1}$. It may happen that $[t+r/2,t+r) \cap I_{k_{1}-1}$ is empty but as before this intersection is at most a single point. If so write $[t+r/2,t+r)\cap I_{k_{1}-1}=\{t_{2}\}$ and let
\[ k_{2}=\min \{ k \in \mathbb{N} : [t+r/2,t+r) \cap I_{k} \neq \varnothing \}.\]
Similarly if, the intersection is non empty, let $(t-r,t-r/2]\cap I_{k_{1}-1}=\{t_{3}\}$ and
\[k_{3}=\min \{k \in \mathbb{N} : (t-r,t-r/2] \cap I_{k} \neq \varnothing \}.\]

For those $1\leq i \leq 3$ for which which we have defined $t_{i}$ let $r_{i}=r-|t-t_{i}|>0$ and define $R_{i} \subset B_{p}((x,t),r)$ to be the rectangle
\[ R_{i}=B_{M}(x,1/2^{k_{i}}) \times (t_{i}-r_{i},t_{i}+r_{i}).\]

We estimate the measure of the region in $B_{p}((x,t),r)$ not covered by the defined rectangles. Notice
\[ B_{p}((x,t),r) \setminus \cup_{i} R_{i} \subset B_{M}(x,1/2^{k_{1}+1}) \times (t-r,t+r).\]
Using that $R_{1} \subset B_{p}((x,t),r)$ and $r_{1}> r/2$ this implies
\[ \frac{ \nu \times \mathcal{L}^{1}( B_{p}((x,t),r) \setminus \cup_{i} R_{i})}{ \nu \times \mathcal{L}^{1} (B_{p}((x,t),r))} \leq \frac{ 2r\nu_{k_{1}+1}}{ r\nu_{k_{1}}} < \varepsilon. \qedhere \]
\end{proof}

Our proof of the theorem will use the density theorem for vertical lines and rectangles to construct paths along which the derivative of $f$ is almost constant. We then show that for most points changes in $f$ are well approximated by the product of the derivative and the change in height.

\begin{lemma}
The density theorem with respect to rectangles holds for $d_{\infty}$ Borel measurable sets in $M \times I$. That is, if $A \subset M \times I$ is $d_{\infty}$ Borel measurable then for almost every $(x,t) \in A$ for all $\varepsilon > 0$ there exists $R>0$ such that $0<u,v<R$ implies
\[ \frac{ \nu \times \mathcal{L}^{1} ( B_{M}(x,u) \times (t-v,t+v) \setminus A)}{2v\nu(B_{M}(x,u))} < \varepsilon. \]
Each $d_{\infty}$ Borel measurable function $g\colon M \times I \to \mathbb{R}$ is approximately continuous almost everywhere. That is, for almost every $(x,t) \in M \times I$ for all $\varepsilon > 0$ there exists $R>0$ such that $0<u,v<R$ implies
\[ \frac{ \nu \times \mathcal{L}^{1} \{(y,s) \in B_{M}(x,u) \times (t-v,t+v) : |g(y,s)-g(x,t)|>\varepsilon\}}{2v\nu(B_{M}(x,u))}<\varepsilon.\]
\end{lemma}

The first part of the lemma follows from Theorem 4 of \cite{Rectangles} by Bruckner and Weiss or can be proven using the density theorem in $M$ and $I$. The second part then follows by applying the density theorem in the usual way.

We can now prove the theorem.

\begin{proof}[Proof of Theorem]

Let $f \colon (M \times I, d_{p}) \to \mathbb{R}$ be Lipschitz. Since $i\colon (M \times I, d_{\infty}) \to (M \times I, d_{p})$ is Lipschitz this implies $f$ is Lipschitz as a function on $(M \times I,d_{\infty})$. It follows easily from Fubini's theorem that $df$ exists almost everywhere in $M \times I$ and can be extended to a $d_{\infty}$ Borel measurable function on $M\times I$. Thus we can use approximate continuity of $df$ along vertical lines and for rectangles. Fix $\varepsilon > 0$. We begin by using the density theorem to construct paths along which $df$ doesn't vary too much.

\begin{claim}
Define, for each $(x,t) \in M \times I$,
\[D(x,t,\varepsilon):=\{(y,s) \in M \times I : | df(y,s)-df(x,t) |\leq\varepsilon \}.\]
For each $k$ the following are $d_\infty$ Borel subsets of $M \times I$:

\begin{multline*}E_k^1:=\{(x,t):\mathcal{L}^{1} \{ s \in (t-r,t+r) : (x,s) \in D(x,t,\varepsilon)\}
\\\geq (2-\varepsilon)r \ \forall \ 0<r<1/k\}\end{multline*}
\begin{multline*}E_k^2:=\{(x,t):\nu \times \mathcal{L}^{1}(B_{M}(x,r) \times (t-s,t+s)\cap E_k^1 \cap D(x,t,\varepsilon))
\\\geq (2-\varepsilon)s\nu(B_M(x,r))\ \forall \ 0<r<1/k\}\end{multline*}
\begin{multline*}E_k^3:=\{(x,t):\mathcal{L}^{1}\{s\in(t-r,t+r):(x,s)\in E_k^2\cap D(x,t,\varepsilon)\}
\\\geq (2-\varepsilon^2) r\ \forall\  0<r<1/k\}.\end{multline*}

Furthermore $\cup_{k=1}^{\infty} E_{k}^{3}$ is a set of full measure in $M\times I$.
\end{claim}

\begin{proof}
For a fixed $q$ let
\[D_q = \{(y,s):|df(y,s)-q|<\varepsilon\}\]
Then for any Borel $A\subset I$ and $r,s,q\in\mathbb R$ the functions
\[(x,t)\mapsto \mathcal{L}^1 \{s\in(t-r,t+r):(x,s)\in D_q \cap A\}\]
and
\[(x,t)\mapsto \nu\times\mathcal{L}^1(B_M(x,r)\times(t-s,t+s)\cap D_q \cap A)\]
are defined on a product of separable metric spaces, Borel measurable in the first variable by Fubini's theorem and continuous in the second and so are Borel measurable as a function of two variables. By taking countable intersections over appropriate $q,r,s\in\mathbb Q$ with $A=M\times I, E_k^1$ and $E_k^2$ shows the measurability of $E_k^1,E_k^2$ and $E_k^3$ respectively.

By various applications  of Fubini's theorem and approximate continuity for vertical lines and rectangles we see that $\cup_{k=1}^\infty E_k^1, \cup_{k=1}^\infty E_k^2$ and hence $\cup_{k=1}^\infty E_k^3$ are of full measure.
\end{proof}

Now fix $(x_{0},t_{0}) \in E_{k}^{3}$ for some $k$ with $t_{0} \notin \cup_{l=1}^{\infty} I_{l}$. Choose $R<1/k$ small enough that $(t_{0}-R,t_{0}+R)\cap I_{s} \neq \varnothing$ implies $1/2^{s}<1/k$ and $\nu_{s+1}/\nu_{s} < \varepsilon /2$. Fix $0<r<R$.

As in the lemma we can find at most three, not necessarily disjoint, rectangles $R_{i} \subset B_{p}((x_{0},t_{0}),r)$ of the form
\[ R_{i}=B_{M}(x_{0},1/2^{k_{i}}) \times (t_{i}-r_{i},t_{i}+r_{i})\]
where $t_{i} \in I_{k_{i}}$, $r_{i}=r-|t_{0}-t_{i}|>0$, $k_{1}= \min \{ k \in \mathbb{N} : (t_{0}-r/2,t_{0}+r/2) \cap I_{k} \neq \varnothing \}$ and $k_{2}, k_{3} < k_{1}$, if defined, such that
\[ \frac{\nu \times \mathcal{L}^{1}( B_{p}((x_{0},t_{0}),r) \setminus \cup_{i} R_{i})}{\nu \times \mathcal{L}^{1} (B_{p}((x_{0},t_{0}),r))}  < \varepsilon.\]

The vertical paths we use will, for each $i$, join $(x_{0},t_{0})$ to $(x_{0},t_{i})$ and, for most points $(x,t) \in R_{i}$, $(x,t)$ to $(x,t_{i})$. Thus we need to consider points $(x,t) \in R_{i}$ whose distance to $(x_{0},t_{0})$ is given by the length of these paths.

Let
\[G_{i}=\{(x,t) \in R_{i} \colon d_{p}((x,t),(x_{0},t_{0}))=|t-t_{i}|+|t_{i}-t_{0}|\}\]
and define
\[ W_{i}= (B_{M}(x_{0},1/2^{k_{i}})\setminus B_{M}(x_{0},1/2^{k_{i}+1})) \times (t_{i}-r_{i},t_{i}+r_{i}) \subset R_{i}.\]
Notice that
\[ \nu \times \mathcal{L}^{1} (R_{i}\setminus W_{i}) = 2r_{i}\nu_{k_{i}+1} \leq \varepsilon r_{i}\nu_{k_{i}}\leq \varepsilon \nu \times \mathcal{L}^{1}(R_{i}).\]
and $\cup_{i} W_{i} \subset \cup_{i} G_{i}$. Thus 
\[\nu \times \mathcal{L}^{1} (\cup_{i} R_{i} \setminus \cup_{i} G_{i}) \leq \nu \times \mathcal{L}^{1} (\cup_{i} R_{i} \setminus \cup_{i} W_{i}) \leq 3\varepsilon \nu \times \mathcal{L}^{1}(B_{p}((x_{0},t_{0}),r).\]

Now fix $i$. We show that for most points in $G_{i}$ the change in $f$ is well approximated by the product of the derivative and the change in height. 

\begin{claim}
Let
\[L(x_{0},t_{0},\varepsilon):=\{ (x,t):\frac{|f(x,t)-f(x_{0},t_{0})-df(x_{0},t_{0})(t-t_{0})|}{d_{p}((x_{0},t_{0}),(x,t))} > (4 \mathrm{Lip} f +4)\varepsilon\}.\]
Then $\nu \times \mathcal{L}^{1} (G_{i}\cap L(x_{0},t_{0},\varepsilon)) \leq \varepsilon \nu \times \mathcal{L}^{1}(R_{i})$.
\end{claim}

\begin{proof}

There are two cases depending on the shape and position of $R_{i}$ in relation to $(x_{0},t_{0})$.

Suppose $r_{i}\leq \varepsilon |t_{i}-t_{0}|$ and let $(x,t) \in G_{i}$ so that $|t_{i}-t_{0}| \leq d_{p}((x_{0},t_{0}),(x,t))$. We calculate, since $d_{p}((x_{0},t_{i}),(x,t_{i}))=0$ implies $f(x_{0},t_{i})=f(x,t_{i})$,
\begin{align*}
|f(x,t)-f(x_{0},t_{0})&-df(x_{0},t_{0})(t-t_{0})|\\
& \leq |f(x,t)-f(x,t_{i})|+|df(x_{0},t_{0})(t-t_{i})| \\
& \qquad+ |f(x_{0},t_{i})-f(x_{0},t_{0})-df(x_{0},t_{0})(t_{i}-t_{0})|\\
& \leq 2 \mathrm{Lip} f |t-t_{i}| + \left| \int_{t_{0}}^{t_{i}} (df(x_{0},s)-df(x_{0},t_{0})) ds \right| \\
& \leq 2 \mathrm{Lip} f r_{i} + \varepsilon |t_{i}-t_{0}| + 2 \mathrm{Lip} f \varepsilon |t_{i}-t_{0}|\\
& \leq (4 \mathrm{Lip} f + 1) \varepsilon d_{p}((x_{0},t_{0}),(x,t))
\end{align*}
using the fundamental theorem of calculus along vertical lines and the fact $(x_{0},t_{0}) \in E_{k}^{3} \subset E_{k}^{1}$ implies
\[ \mathcal{L}^{1} \{s \in (t_{0}-|t_{0}-t_{i}|,t_{0}+|t_{0}-t_{i}|) : (x_{0},s) \not\in D(x_{0},t_{0},\varepsilon)\} \leq \varepsilon |t_{i}-t_{0}|.\]

Suppose $r_{i}> \varepsilon |t_{i}-t_{0}|$. Then the rectangle $R_{i}$ is large enough to be well approximated by the rectangles from our use of the density theorem. 

Since $(x_{0},t_{0}) \in E_{k}^3$ we can find $t_{i}'$ with $|t_{i}'-t_{i}|<\varepsilon^{2} |t_{i}-t_{0}|<\varepsilon r_i$, $|df(x_{0},t_{i}')-df(x_{0},t_{0})|<\varepsilon$ and $(x_{0},t_{i}')\in E_{k}^{2}$. Thus if we define
\[ A_{i}=B_{M}(x_{0},1/2^{k_{i}}) \times (t_{i}'-r_{i},t_{i}'+r_{i})\]
then
\[ \nu \times \mathcal{L}^{1}(A_{i} \cap E_{k}^{1} \cap D(x_{0},t_{i}',\varepsilon)) \geq (2-\varepsilon) r_{i} \nu_{k_{i}}.\]

Notice $D(x_{0},t_i',\varepsilon) \subset D(x_{0},t_0,2\varepsilon)$. Since $G_{i} \subset R_{i}$ and $R_{i}$ is well approximated by $A_{i}$,
\begin{align*}
 \nu & \times \mathcal{L}^{1} (G_{i}\setminus(E_{k}^{1}\cap D(x_{0},t_{0},2\varepsilon)))\\
& \leq \nu \times \mathcal{L}^{1} (A_{i}\setminus (E_{k}^{1}\cap D(x_{0},t_{i}',\varepsilon))) + \nu \times \mathcal{L}^{1} (R_{i}\setminus A_{i})\\
& \leq  \varepsilon r_{i} \nu_{k_{i}} + \nu_{k_{i}} |t_{i}'-t_{i}|\\
& \leq  \varepsilon r_{i} \nu_{k_{i}} + \varepsilon r_{i} \nu_{k_{i}}\\
& =  \varepsilon \nu \times \mathcal{L}^{1}(R_{i}).
\end{align*}

Now fix $(x,t) \in G_{i}\cap E_{k}^{1} \cap D(x_{0},t_{0},2\varepsilon)$ and notice $|df(y,s)-df(x_{0},t_{0})|>3\varepsilon$ implies $(y,s) \not\in D(x,t,\varepsilon)$. We calculate,
\begin{align*}
|f & (x,t)-f(x_{0},t_{0})-df(x_{0},t_{0})(t-t_{0})| \\
& \leq \left| \int_{t_{0}}^{t_{i}} (df(x_{0},s)-df(x_{0},t_{0})) ds \right| + \left| \int_{t_{i}}^{t} (df(x,s)-df(x_{0},t_{0})) ds \right| \\
& \leq (2\mathrm{Lip} f +1)\varepsilon d_{p}((x_{0},t_{0}),(x,t)) + 3\varepsilon |t-t_{i}| + 2\varepsilon \mathrm{Lip} f |t-t_{i}|\\
&\leq (4\mathrm{Lip} f+4) \varepsilon d_{p}((x_{0},t_{0}),(x,t)). 
\end{align*}
The estimation of the first integral is the same as earlier. Estimation of the second uses the fact that $(x,t) \in E_{k}^{1}$ as for $(x_0,t_0)$ before.
\end{proof}

To conclude we estimate
\begin{align*}
 \nu & \times \mathcal{L}^{1} (B_{p}((x_{0},t_{0}),r)\cap L(x_{0},t_{0},\varepsilon)) \\
& \leq  \nu\times\mathcal{L}^{1} (B_{p}((x_{0},t_{0}),r)\setminus \cup_{i} G_{i}) + \sum_{i} \nu\times\mathcal{L}^{1} (G_{i}\cap L(x_{0},t_{0},\varepsilon))\\
& \leq  \nu \times \mathcal{L}^{1}(B_{p}((x_{0},t_{0}),r) \setminus \cup_{i} R_{i}) + \nu \times \mathcal{L}^{1} (\cup_{i} R_{i} \setminus \cup_{i} G_{i}) + \sum_{i} \varepsilon \nu \times \mathcal{L}^{1}(R_{i})\\
& \leq 7\varepsilon \nu\times\mathcal{L}^{1}(B_{p}((x_{0},t_{0}),r)).
\end{align*}

We have shown that given $\varepsilon > 0$ for almost every $(x_{0},t_{0}) \in M \times I$ there exists $R>0$ such that $0<r<R$ implies
\[ \nu \times \mathcal{L}^{1} (B_{p}((x_{0},t_{0}),r)\cap L(x_{0},t_{0},\varepsilon)) \leq 7\varepsilon \nu\times\mathcal{L}^{1}(B_{p}((x_{0},t_{0}),r)).\]
By taking a countable intersection of such points for $\varepsilon_n = 1/n$ we see that $f$ is approximately differentiable almost everywhere.
\end{proof}

\section{Further Discussion}

There are some related questions we can ask about approximate differentiability spaces and non doubling measures. In our example $\nu \times \mathcal{L}^{1}$ was still, at almost every point, doubling on arbitrarily small scales. That is, for almost every $(x_{0},t_{0}) \in M \times I$,
\[ \liminf_{r \downarrow 0} \frac{\nu \times \mathcal{L}^{1}(B_{p}((x_{0},t_{0}),4r))}{\nu \times \mathcal{L}^{1}(B_{p}((x_{0},t_{0}),r))} < \infty.\]
It is not clear whether there is an approximate differentiability space $X$ with measure $\mu$ that is, at almost every point, non doubling on all small scales. That is, for some enlargement factor $C$,
\[ \lim_{r \downarrow 0} \frac{\mu(B(x_{0},Cr))}{\mu(B(x_{0},r))} = \infty\]
for almost every $x \in X$.
If we adjust the definition of $M$ by splitting each island into a very large number of subislands and define a different equivalence relation on $M\times I$ so that at a jump level of order $k$ we identify corresponding points in distinct islands of level $k+1$ that are inside the same island of level $k$ then the measure $\nu \times \mathcal{L}^{1}$ is non doubling on all small scales. However, as we identify fewer points (only finitely many at each level) there seems no obvious way to use the density theorem in horizontal directions to construct paths joining points along which the derivative is approximately constant. Thus, it is not clear whether this alternate construction also gives an approximate differentiability space.

We can also consider an alternative definition of approximate differentiability space in which the approximate limit defining the derivative is defined in a Lipschitz invariant way. That is, in a complete separable metric space $X$ with locally finite measure $\mu$ we say $g\colon X \to \mathbb{R}$ has invariant approximate limit $l \in \mathbb{R}$ at $x_{0} \in X$ if for for every $C, \varepsilon > 0$,
\[ \lim_{r\downarrow 0} \frac{\mu \{x \in B(x_{0},Cr): |g(x)-l|>\varepsilon \}}{\mu(B(x_{0},r))}=0.\]
It is not clear whether there is any invariant approximate differentiability space in which the measure is non doubling. There seems no obvious way to construct such a space using density theorems and approximate continuity like before. One can however show the measure cannot be non doubling on all small scales. Precisely, for all $C>0$ and almost every $x_{0}\in X$,
\[ \liminf_{r \downarrow 0} \frac{\mu(B(x_{0},Cr))}{\mu(B(x_{0},r))} < \infty.\]
The proof of this involves using a contradiction argument to construct a Lipschitz function not invariantly approximately differentiable on a set of positive measure. The idea is find a fixed radius $r>0$ such that multiplying $r$ by a fixed factor $\eta >0$ greatly reduces the measure of the corresponding ball centred on most points. We then use a covering theorem to find a disjoint subcollection of the balls with radius $r$ so that each ball in the original collection meets a ball in the subcollection in such a way that the corresponding reduced balls have comparable measure. We may build cones on the reduced balls in the disjoint subcollection, and later repeat the process at smaller scales, as collectively they have small measure. Now around most centres $c$ we can find a ball $B(c,\eta r)$ for which the expanded ball $B(c, r)$ contains a ball $B(c',\eta r)$ on which a cone is defined and $B(c,\eta r)$ and $B(c',\eta r)$ have comparable measure. By either adding cones or not at smaller and smaller scales, similar to in the proof porous sets in differentiability spaces are null, we construct a Lipschitz function which is not invariantly approximately differentiable on a set of positive measure.

\bibliographystyle{amsplain}

\end{document}